\documentclass[12pt]{amsart}
\usepackage{latexsym, amsmath, amssymb,amsthm,amsopn,amsfonts}
\usepackage{version}
\usepackage{epsfig,graphics,color,graphicx,graphpap,dsfont}
\usepackage{amssymb}

\ifx\pdfoutput\undefined
  \DeclareGraphicsExtensions{.eps}
\else
  \ifx\pdfoutput\relax
    \DeclareGraphicsExtensions{.eps}
  \else
    \ifnum\pdfoutput>0
      \DeclareGraphicsExtensions{.pdf}
    \else
      \DeclareGraphicsExtensions{.eps}
    \fi
  \fi
\fi

\newcommand{\D}{{\mathcal D}}

\renewcommand{\l}{\lambda}

\newcommand{\LV}{\left|}
\newcommand{\RV}{\right|}
\newcommand{\LN}{\left\|}
\newcommand{\RN}{\right\|}
\newcommand{\LB}{\left[}
\newcommand{\RB}{\right]}
\newcommand{\LC}{\left(}
\newcommand{\RC}{\right)}

\newcommand{\R}{\mathbb{R}}
\newcommand{\vertiii}[1]{{\left\vert\kern-0.25ex\left\vert\kern-0.25ex\left\vert #1
    \right\vert\kern-0.25ex\right\vert\kern-0.25ex\right\vert}}

\theoremstyle{plain}

\newtheorem{thm}{Theorem}[section]
\newtheorem{prop}{Proposition}[section]

\newtheorem{lem}[prop]{Lemma}

\theoremstyle{definition}

\newtheorem{rem}{Remark}[section]
\newtheorem{defn}[prop]{Definition}

\numberwithin{equation}{section}

\def\squarebox#1{\hbox to #1{\hfill\vbox to #1{\vfill}}}

\usepackage{amsxtra}

\def\ba{\begin{array}}
\def\ea{\end{array}}
\def\bea{\begin{eqnarray}}
\def\eea{\end{eqnarray}}
\def\beas{\begin{eqnarray*}}
\def\eeas{\end{eqnarray*}}
\def\bi{\begin{itemize}}
\def\ei{\end{itemize}}

\def\g{\gamma}

\def\l{\lambda}

\def\si{\sigma}

\def\o{\omega}

\def\D{\Delta}

\def\O{\Omega}

\def\hR{\mathbb{R}}

\def\({\textnormal{(}}
\def\){\textnormal{)}}
\def\[{[\neg[}
\def\]{]\neg]}
\def\lan{\langle}
\def\ran{\rangle}

\def\q{\quad}

\def\neg{\negthinspace}

\def\hb{\hbox}

\def\pa{\partial}

\def\b1{{\bf 1}}

\title[NSE with randomized data]{Almost sure existence of Navier-Stokes Equations with randomized data in the whole space *}

\author[R.M. Chen]
{Robin Ming Chen}
\address{Robin Ming Chen\newline
Department of Mathematics\\
University of Pittsburgh\\
Pittsburgh, PA 15260} \email{mingchen@pitt.edu}
\author[D. Wang]
{Dehua Wang}
\address{Dehua Wang\newline
Department of Mathematics\\
University of Pittsburgh\\
Pittsburgh, PA 15260}
\email{dwang@math.pitt.edu}
\author[S. Yao]
{Song Yao}
\address{Song Yao\newline
Department of Mathematics\\
University of Pittsburgh\\
Pittsburgh, PA 15260}
\email{songyao@pitt.edu}
\author[C. Yu]
{Cheng Yu}
\address{Cheng Yu\newline
Department of Mathematics\\
University of Pittsburgh\\
Pittsburgh, PA 15260}
\email{chy39@pitt.edu}

\thanks{* This is a research note from an internal working seminar.}

\begin{document}

\begin{abstract}
This short note is concerned with the almost sure existence of weak solutions to the Navier Stokes equations with randomized data, and provides an outline to extend the periodic domain case in \cite{NPS} to the whole space case.
\end{abstract}
\maketitle

\section{Introduction}\label{sec_intro}
The goal of this note is to apply the main idea of \cite{NPS} to the incompressible Navier-Stokes equations in $\mathbb{R}^d$ with randomized initial data. The Navier-Stokes equations with pressure term being eliminated are given by
\begin{equation}\label{NSE_nopressure}
\left\{\begin{split}
& \pa_t \vec{u} = \Delta \vec{u} - \mathbb{P} \nabla \cdot (\vec{u} \otimes \vec{u}), \quad x\in \mathbb{R}^d,\ t>0, \\
& \nabla \cdot \vec{u} = 0, \\
& \vec{u}(x, 0) = \vec{f}(x), 
\end{split}
\right.
\end{equation}
where $\vec{u}$ is the fluid velocity, $\vec{f}$ is a divergence free vector field, and $\mathbb{P}$ is the Leray-Hopf projection onto the divergence free vector fields defined by
\begin{equation}\label{LerayP}
\mathbb{P} = \text{Id} - \nabla \Delta^{-1} \nabla \cdot.
\end{equation}
Here we take the viscosity to be 1.

The local/global well-posedness of the Navier-Stokes equations with weak/smooth initial data has been studied extensively, in both mild and weak solution settings. In this note, we consider the global weak solutions to \eqref{NSE_nopressure}, with initial data below certain regularity. The central difficulty of the supercritical case is that the linear part does not provide enough regularity/integrability required to estimate the nonlinear term. In \cite{BT1}, Burq and Tzvetkov used the randomization method to proved the local-wellposedness for the supercritical wave equations in the mild solution setting. In \cite{NPS}, Nahmod, Pavlovi\'c, and Staffilani also successfully employed the randomization technique to obtain the first global existence result on $\mathbb{T}^d$ in the weak solution setting. In this note, we extend the idea of \cite{NPS} to $\mathbb{R}^d$.

\subsection{Randomization}\label{subsec_random}

\medskip

Now we introduce our randomization setup for the initial data. First we follow the idea in \cite{ZF2} to construct a partition of the whole Fourier space. Define the rings
\begin{equation}\label{partition}
A_n = \big\{ \xi\in \mathbb{R}^d: \ (n-1)^{1/d} \leq |\xi| < n^{1/d} \big\}, \ \ n\in \mathbb{N}.
\end{equation}
Then we know that
\begin{equation*}
A_m \cap A_n = \emptyset \text{ for } m\neq n, \quad \mathbb{R}^d = \underset{n\in \mathbb{N}}{\cup} A_n,
\quad \text{and }  |A_n| \lesssim 1.
\end{equation*}

\begin{defn}\label{def_random}
Let $\{ l_n(\omega) \}_{n\in \mathbb{N}}$ be a sequence of real, independent, 0-mean random variables on a probability space $(\Omega, \mathcal{M}, P)$ with associated sequence of distributions $\{ \mu_n \}_{n\in \mathbb{N}}$ satisfying the property
\begin{equation}\label{cond1}
\exists\ c>0: \ \ \forall\ \gamma\in \mathbb{R}, \ \ \forall\ n\geq 1, \ \ \LV \int_\mathbb{R} e^{\gamma x}\ d\mu_n(x)\RV \leq e^{c\gamma^2}.
\end{equation}
.

For $\vec{f}\in (H^s(\mathbb{R}^d))^d$, let $\tilde{\Delta}_n$ be the Fourier projection operator given as
\begin{equation}\label{Fourierprojection}
\tilde\Delta_n \vec{f} = \mathcal{F}^{-1} \LC \hat{\vec{f}} (\xi) \chi_{A_n} \RC.
\end{equation}
Consider the map from $(\Omega, \mathcal{M})$ to $(H^s(\mathbb{R}^d))^d$ equipped with the Borel sigma algebra, defined by
\begin{equation}\label{randomization}
\omega\mapsto \vec{f}^{\omega}, \quad \vec{f}^\omega (x) = \sum_{n\in \mathbb{N}} l_n(\omega) \tilde\Delta_n \vec{f},
\end{equation}
and call such a map randomization.
\end{defn}

One can check that such a map is measurable and $\vec{f}^\omega \in L^2(\Omega; (H^s(\mathbb{R}^d))^d)$, and hence defines an $(H^s(\mathbb{R}^d))^d$-valued random variable.

\subsection{Main Results}

We follow the usual definition for the weak solution to the Navier-Stokes equations \eqref{NSE_nopressure} (c.f. \cite{NPS}). Our main results are
\begin{thm}\label{thm_exist}
For any $T>0$. Let $0 < s < \left\{\begin{array}{l}\frac{1}{2}, \q \hb{for } d = 2, \\\\
 {1\over4}, \q \hb{for } d = 3
 \end{array}\right.$. For a fixed $\vec{f}\in (H^{-s}(\R^d))^d$, $\nabla\cdot \vec{f} = 0$, and of mean zero. Let $\vec{f}^\omega$ be the randomization of $\vec{f}$ as defined in Definition \ref{def_random}. Then for almost all $\o \in \O$ there is a weak solution $\vec{u}$ to the Cauchy problem \eqref{NSE_nopressure} with initial data $\vec{f}^\omega$. Moreover $\vec{u}$ is of the form
\begin{equation}\label{solnsplit}
\vec{u} = e^{t\Delta} \vec{f}^\omega + \vec{w}
\end{equation}
where $\vec{w}\in L^\infty([0, T]; (L^2(\R^d)^d) \cap L^2([0, T]; (\dot{H}^1(\R^d))^d)$. When $d = 2$, such a weak solution is unique.

\end{thm}


In Section \ref{sec_diff-eqn} we introduce the difference equation for the fluctuation $\vec{w}$ and build up the approximation scheme via the Friedrich's method, and derive {\it a priori} energy estimates to obtain global weak solutions. In Section \ref{sec_asexist} we give both the deterministic and probabilistic estimates on the heat flow with the randomized data, and consequently obtain the almost sure existence.

\section {Approximation scheme}\label{sec_diff-eqn}

\subsection{Equations for the fluctuation}
Recall from \eqref{solnsplit} that
\begin{equation*}
\vec{u} = e^{t\Delta} \vec{f}^\omega + \vec{w},
\end{equation*}
the equation for the fluctuation $\vec{w}$ is
\begin{equation}\label{nonlinearpart}
\left\{\begin{split}
& \pa_t \vec{w} = \Delta \vec{w} - \mathbb{P} \nabla \cdot (\vec{w} \otimes \vec{w}) - \mathbb{P} \nabla \cdot (\vec{w} \otimes \vec{g}) - \mathbb{P} \nabla \cdot (\vec{g} \otimes \vec{w}) - \mathbb{P} \nabla \cdot (\vec{g} \otimes \vec{g}), \\
& \nabla \cdot \vec{w} = 0, \\
& \vec{w} (x, 0) = 0,
\end{split}
\right.
\end{equation}
where $\vec{g} = e^{t\Delta}\vec{f}^\omega$. We hence look for the weak solution to \eqref{nonlinearpart} which is defined in a similar way as in \cite{NPS}.
We have the following global existence of weak solutions to system \eqref{nonlinearpart}.
\begin{thm}\label{thm_existapprox}
For any $T>0$, let $\lambda>0$, $\gamma<0$ and $\left\{\begin{array}{ll}0<s, & \text{ if } d=2 \\ 0<s<{1\over4} & \text{ if } d = 3 \end{array}\right.$ be given. Let $\vec{g}$ be a divergence free vector field and satisfy
\begin{equation}\label{condg}
\begin{split}
&\|\vec{g}\|_{L^2}\lesssim (1+\frac{1}{t^{s\over2}}),
\\&\|\nabla^k \vec{g}\|_{L^{\infty}}\lesssim \LC \max\{t^{-1},t^{-(k+s+\frac{d}{2})}\}\RC^{\frac{1}{2}},\;\;\text{ for}\;\; k=0,,1
\end{split}
\end{equation}
and
\begin{equation}\label{condtg}\left\{\begin{array}{l}
\|t^{\gamma}\vec{g}(x,t)\|_{L^{4}([0,T];L^4_x)}\leq \lambda,\quad \text{ if } d=2,\\\\
\LN t^{\gamma}\LB I + (-\Delta)^{\frac{1}{4}} \vec{g} \RB\RN_{L^2([0,T]; L^6_x)} \\\\ \quad +\LN t^{\gamma}\LB I + (-\Delta)^{\frac{1}{4}} \RB \vec{g}\RN_{L^{8\over3}([0,T]; L^{8\over3}_x)}+\|t^{\gamma}\vec{g}\|_{L^8([0,T]; L^8_x)}\leq \lambda, \text{ if } d=3.
\end{array}\right.
\end{equation}
Then there exists a weak solution $\vec{w}$ to the system \eqref{nonlinearpart}.
\end{thm}

\subsection{Approximation} Since now the problem is posed in the whole space, we will use the classical Friedrich's approximation scheme. To this end, we first  introduce an operator $J_n$ as follows
 \begin{equation*}
 J_nh(x):=\mathcal{F}^{-1}(\chi_{B_n}(\xi)\hat{h}(x)),
 \end{equation*}
 where $\mathcal{F}^{-1}$ denotes the inverse Fourier transform, $B_n$ is the ball of radius $n$ centered at the origin, and $\hat{h}$ denotes the Fourier transform in the space variables.

The approximation we use is the following
\begin{equation}\label{approximationofnonlinearpart}
\left\{\begin{split}
& \pa_t \vec{w}_n = \Delta J_n\vec{w}_n - J_n\mathbb{P} \nabla \cdot (J_n\vec{w}_n \otimes J_n\vec{w}_n) - J_n\mathbb{P} \nabla \cdot (J_n\vec{w}_n \otimes J_n\vec{g}) -
\\&\;\;\;\;\;\;\;\;\;\;J_n\mathbb{P} \nabla \cdot (J_n\vec{g}\otimes J_n\vec{w}_n) - J_n\mathbb{P} \nabla \cdot (J_n\vec{g} \otimes J_n\vec{g}), \\
& \nabla \cdot \vec{w}_n = 0, \\
& \vec{w}_n (x, 0) = J_n(\vec{w}(0))=0.
\end{split}
\right.
\end{equation}

Following a standard ODE approach similar to the one in \cite{NPS} we have

\begin{lem}\label{lem_approx}
Let $\lambda>0$, $\gamma<0$ and $\left\{\begin{array}{ll}0<s, & \text{ if } d=2 \\ 0<s<{1\over4} & \text{ if } d = 3 \end{array}\right.$ be given. Let $\vec{g}$ be a divergence free vector field and satisfies \eqref{condg} and \eqref{condtg}. For any $n>0$, there exists some $\delta>0$ such that system \eqref{approximationofnonlinearpart} admits a unique solution in $X_\delta = L^\infty([0, \delta]; L^2(\mathbb{R}^d)) \cap L^2([0, \delta]; \dot{H}^1(\mathbb{R}^d))$.
\end{lem}

\begin{rem}\label{rem_condJg}
Note that $J_n$ is a bounded operator on $L^p$ for all $n\in \mathbb{N}$. Hence conditions \eqref{condg} and \eqref{condtg} hold for all $J_ng$ if they hold for $\vec{g}$. Also we see that $J_n$ commutes with $\text{div}$, so $J_n\vec{g}$ is also divergence free. Such properties have been explored in for instance \cite{Mas}.
\end{rem}

Note that $J_n^2=J_n$. It is easy to check that $J_n\vec{w}_n$ is also a solution to \eqref{approximationofnonlinearpart}. This implies $J_n\vec{w}_n=\vec{w}_n$ by the uniqueness and hence, one can remove all the operator $J_n$ in front of $\vec{w}_n$, and only keep those in front of nonlinear parts. Thus, we arrive at the following new system:
\begin{equation}\label{newapproximation}
\left\{\begin{split}
& \pa_t \vec{w}_n = \Delta \vec{w}_n - J_n\mathbb{P} \nabla \cdot (\vec{w}_n \otimes \vec{w}_n) - J_n\mathbb{P} \nabla \cdot (\vec{w}_n \otimes J_n\vec{g}) -
 \\&\;\;\;\;\;\;\;\;\;J_n\mathbb{P} \nabla \cdot (J_n\vec{g} \otimes \vec{w}_n) -J_n \mathbb{P} \nabla \cdot (J_n\vec{g} \otimes J_n\vec{g}), \\
& \nabla \cdot \vec{w}_n = 0, \\
& \vec{w}_n (x, 0) = 0.
\end{split}
\right.
\end{equation}
The key step now is to show that $\delta$ can be extended to any time $T>0$ and that we have some local-in-time estimates which are uniform in $n$. Applying Aubin-Lions Lemma, one can pass to the limit and recover a solution to the initial system \eqref{NSE_nopressure}. In order to arrive at this goal, we need some {\it a priori} energy estimates for the difference equation \eqref{nonlinearpart}.

\subsection{Energy estimates}
The energy functional associated to \eqref{nonlinearpart} is
\begin{equation*}
E(\vec{w})=\int_{\mathbb{R}^d}|\vec{w}|^2d\,x+ \int_0^t\int_{\mathbb{R}^d}|\nabla  \vec{w}|^2d\,x\,d\,\tau.
\end{equation*}
We have
\begin{lem}
\label{energyestimate}
Under the same assumptions as in Lemma \ref{lem_approx}, let $\vec{w}\in X_T$ be a solution to the approximation system \eqref{newapproximation} on $[0, T]$ for any fixed $T>0$. Then
\begin{align}
&E(\vec{w})(t)\leq C(T,\lambda,\g, s),\;\;\;\;\text{ for all }t\in [0,T], \label{energyE}\\
& \LN\frac{d}{dt}\vec{w}\RN_{L^{4\over d}([0, T]; H^{-1}_x)}\leq C(T,\lambda,\g, s). \label{energydw}
\end{align}
\end{lem}
\begin{rem}\label{rem_g}
The goal of the lemma is to give a uniform (in $n$) {\it a priori} energy estimate of the approximating solutions of \eqref{newapproximation}. The standard energy estimate together with the linear estimates in $\vec{g}$ provides the uniform control on the energy for ``large time". However the uniform ``short time" estimate is much more delicate.
\end{rem}

\begin{proof}
We follow the idea from \cite{NPS}. That is, we use two different formulations of system \eqref{nonlinearpart} to handle the case when $t$ is near zero and when $t$ is away from zero respectively. For details, please refer to Theorem 5.1 in \cite{NPS}.
\end{proof}

\subsection{Proof of Theorem \ref{thm_existapprox}} Applying Lemma \ref{energyestimate} to the approximate system \eqref{newapproximation} implies that the $L^2$ norm of $\vec{w}_n$ is uniformly bounded and hence, $\delta$ can be taken up to any time $T$. With Lemma \ref{energyestimate}, the Aubin-Lions Lemma, and together with the fact that $J_ng$ converges to $g$ in $L^p$, one allows us to pass to the limit in \eqref{newapproximation}. This yields the global existence of a solution $\vec{w}$ to \eqref{nonlinearpart}.

\subsection{Uniqueness in $2D$} Furthermore, similar as for the classical Navier-Stokes equations, we have the following uniqueness result of the difference equation \eqref{nonlinearpart} when $d=2$. The proof follows quite similarly to Theorem 7.1 in \cite{NPS} and hence we omit it.
\begin{thm}\label{thm_uniqueness2d}
Assume that $\vec{g}$ satisfies the conditions in Theorem \ref{thm_existapprox} and $d = 2$. Then for any $T>0$ system \eqref{nonlinearpart} admits a unique weak solution in $L^2([0, T]; V) \cap L^\infty([0, T]; H)$.
\end{thm}

\section{Almost Sure Existence}\label{sec_asexist}

First we follow the same idea as in \cite{BT1} and \cite{NPS} to derive estimates on the linear evolution part of the randomized data, which involve both the deterministic and probabilistic aspects.

\subsection{Deterministic heat flow estimates}

\begin{lem}\label{lem_heatflow}
Let $0<s<1$, $k \in \mathbb{N} \cup \{0\}$, and $\vec{u}_{\vec{f}^\omega} = e^{t\Delta} \vec{f}^\omega$. If $\vec{f}^\omega \in (H^{-s}(\mathbb{R}^d))^d$ then
\begin{align}
\LN \nabla^k \vec{u}_{\vec{f}^\omega} (\cdot, t) \RN_{L^2_x} & \lesssim \LC 1 + t^{-{s+k\over2}} \RC \|\vec{f}\|_{H^{-s}},\label{linearest1}\\
\LN \nabla^k \vec{u}_{\vec{f}^\omega} (\cdot, t) \RN_{L^\infty_x} & \lesssim \LC \max\{ t^{-1}, t^{-(k+s+{d\over2})} \} \RC \|\vec{f}\|_{H^{-s}}. \label{linearest2}
\end{align}
\end{lem}
\begin{proof}
The proof follows quite similarly to the proof of Lemma 3.1 in \cite{NPS} and hence we omit it here.
\end{proof}

\subsection{Averaging effects: probabilistic heat flow estimates}

Now we give the improved $L^p$ estimates for the linear evolution.
 \begin{prop}\label{prop_averaging}
 Let $T > 0$ and $s \ge 0$. Let $r  \ge  p  \ge  q  \ge  2$, $\gamma\in\mathbb{R}$, and $\si  \ge  0$  be such that
 \begin{equation}\label{conds}
(\si + s - 2\g  ) q < 2.
\end{equation}
Then there exists $C_T > 0$ such that for every $ \vec{f} \in \big(H^{-s}(\hR^d)\big)^d$,
 \bea \label{eq:a002}
 \big\| t^\gamma (-\D)^{\frac{\si}{2}}e^{t \D} \vec{f}^\o \big\|_{L^r(\O; L^q([0,T];L^p_x))}
 \le C_T \|\vec{f}\|_{H^{-s}} ,
 \eea
 where $C_T= C_T(p,q,r,\si,s)$.

 Moreover, thanks to the Bienaym\'e-Tchebichev inequality, if we set
  \bea \label{eq:a003}
  E_{\l,T,\vec{f},\si,p} = \big\{ \o \in \O:
  \big\| t^\gamma(-\D)^{\frac{\si}{2}}e^{t \D} \vec{f}^\o \big\|_{  L^q([0,T];L^p_x)}  \ge \l \big\},
  \eea
  then there exists $C_1, C_2 > 0$ such that for every $\l > 0$ and for every $\vec{f} \in \big(H^{-s}(\hR^d)\big)^d$,
  \bea  \label{eq:a004}
   P \big(E_{\l,T,\vec{f},\si,p}\big) \le C_1 \exp \left[ -C_2 \frac{\l^2}{C_T \|\vec{f}\|^2_{H^{-s}}} \right] .
  \eea
 \end{prop}
\begin{proof}
For $t \ne 0$, let $\vec{h}(x) = \lan \sqrt{- \D} \ran^{-\frac{s}{2}} \vec{f}(x)$.
From \eqref{randomization} we have
  \begin{equation}\label{eq:a005}
  \begin{split}
   t^\gamma(-\D)^{\frac{\si}{2}} e^{t \D} \vec{f}^\o (x)
  &=  t^\gamma(-\D)^{\frac{\si}{2}} \lan \sqrt{- \D} \ran^{\frac{s}{2}}
  e^{t \D} \lan \sqrt{- \D} \ran^{-\frac{s}{2}} \vec{f}^\o (x)   \\
  &= \sum_{n\in \mathbb{N}} t^\gamma(-\D)^{\frac{\si}{2}} \lan \sqrt{- \D} \ran^{\frac{s}{2}}e^{t \D} l_n(\omega)\tilde{\D}_n \vec{h}  \\
 &= \sum_{n\in \mathbb{N}} l_n(\omega)\tilde{\D}_n\LB t^\gamma(-\D)^{\frac{\si}{2}} \lan \sqrt{- \D} \ran^{\frac{s}{2}}e^{t \D}  \vec{h} \RB.
  \end{split}
  \end{equation}

Applying Lemma 3.1 of \cite{BT1} and Minkowski's inequality we have
\begin{align*}
\LN t^\gamma (-\D)^{\frac{\si}{2}}e^{t \D} \vec{f}^\o \RN_{L^r(\O; L^q([0,T];L^p_x))} &\lesssim \sqrt{r} \LN \LC \sum_{n\in\mathbb{N}} \LB \tilde{\D}_n t^\gamma(-\D)^{\frac{\si}{2}} \lan \sqrt{- \D} \ran^{\frac{s}{2}}e^{t \D}  \vec{h}  \RB^2 \RC^{1/2} \RN_{L^q([0,T];L^p_x)}\\
&\lesssim \sqrt{r} \LC \sum_{n\in\mathbb{N}}\LN \tilde{\D}_n t^\gamma(-\D)^{\frac{\si}{2}} \lan \sqrt{- \D} \ran^{\frac{s}{2}}e^{t \D}  \vec{h}   \RN^2_{L^q([0,T];L^p_x)} \RC^{1/2}.
\end{align*}
Recall the Haussdorf-Young inequality
\begin{equation}\label{Becknerinequ}
\LN \hat{f}(\xi) \RN_{L^{p'}_\xi} \lesssim \|f(x)\|_{L^p_x}, \quad \text{for } 1\leq p \leq 2, \text{ and } {1\over p} + {1\over p'} = 1.
\end{equation}

Note that by assumption, $p\geq 2$. So the application of the dual form of \eqref{Becknerinequ} implies that
\begin{align*}
\LN t^\gamma (-\D)^{\frac{\si}{2}}e^{t \D} \vec{f}^\o \RN_{L^r(\O; L^q([0,T];L^p_x))} &\lesssim \sqrt{r} \LC \sum_{n\in\mathbb{N}}\LN t^\gamma |\xi|^\si \lan \xi \ran^{s/2} e^{-t|\xi|^2} \hat{h} \chi_{A_n}   \RN^2_{L^q([0,T];L^{p'}_\xi)} \RC^{1/2}.
\end{align*}
Using H\"{o}lder and Minkowski inequalities the above estimate further yields
\begin{align*}
& \LN t^\gamma (-\D)^{\frac{\si}{2}}e^{t \D} \vec{f}^\o \RN _{L^r(\O; L^q([0,T];L^p_x))} \\
& \quad \lesssim \sqrt{r} \sum_{n\in\mathbb{N}} \LN \tilde{\D}_n h \RN_{L^2} \LC \int^T_0 t^{q\gamma} \LN |\xi|^\si \lan \xi \ran^{s/2} e^{-t|\xi|^2} \RN^q_{L^{2p/(p-2)}_\xi(A_n)} \ dt \RC^{1/q}.
\end{align*}
It follows that
\begin{align*}
|\xi|^\si \lan \xi \ran^{s/2} e^{-t|\xi|^2} & \lesssim t^{-\si/2} \LC t|\xi|^2 \RC^{\si/2} e^{-t|\xi|^2} + t^{-(\si+s)/2} \LC t|\xi|^2 \RC^{(\si+s)/2} e^{-t|\xi|^2}\\
& \lesssim t^{-\si/2} + t^{-(\si+s)/2}.
\end{align*}
So
\begin{equation*}
\LN |\xi|^\si \lan \xi \ran^{s/2} e^{-t|\xi|^2} \RN_{L^{2p/(p-2)}_\xi(A_n)} \lesssim \LC t^{-\si/2} + t^{-(\si+s)/2} \RC |A_n| \lesssim t^{-\si/2} + t^{-(\si+s)/2}.
\end{equation*}

This way we finally obtain
\begin{align}
& \LN t^\gamma (-\D)^{\frac{\si}{2}}e^{t \D} \vec{f}^\o \RN_{L^r(\O; L^q([0,T];L^p_x))} \\
& \quad \lesssim \sqrt{r} \sum_{n\in\mathbb{N}} \LN \tilde{\D}_n h \RN_{L^2} \LC \int^T_0 t^{q\gamma} (t^{-\si/2} + t^{-(\si+s)/2})^q \ dt \RC^{1/q} \nonumber \\
& \quad \lesssim \sqrt{r} \sum_{n\in\mathbb{N}} \LN \tilde{\D}_n h \RN_{L^2} \LB \LC \int^T_0 {1\over t^{({\si\over2} -\gamma)q}} \ dt \RC^{1/q} + \LC \int^T_0 {1\over t^{({\si+s \over 2} - \gamma)q}} \ dt \RC^{1/q} \RB \nonumber \\
& \quad \lesssim \sqrt{r} \LC T^{{1\over q} + {2\gamma - \si \over 2}} + T^{{1\over q} + {2\gamma - (\si + s) \over 2}} \RC \LN h \RN_{L^2} \label{estprop}
\end{align}
provided
\begin{equation*}
\LC {\si\over2} -\gamma \RC q<1, \quad \LC {\si+s \over 2} - \gamma \RC q<1.
\end{equation*}
The above condition is satisfied under the assumption in the proposition. Hence the estimate \eqref{eq:a002} follows from \eqref{estprop}.
\end{proof}

\subsection{Proof of Theorem \ref{thm_exist}}

Recall that in the construction of the weak solution we decomposed $\vec{u}$ as $\vec{u} = e^{t\Delta} \vec{f}^\omega + \vec{w}$, and hence we only need to seek the weak solution $\vec{w}$ to the difference equation \eqref{nonlinearpart}.  As is stated in Theorem \ref{thm_existapprox}, this is guaranteed when condition \eqref{condtg} is satisfied for the linear evolution. One can follow the argument in \cite{NPS} to obtain the almost sure existence result.

\end{document}